\documentclass{amsart}

\usepackage{amsmath,amsfonts,amssymb,amsthm,amscd,comment,euscript}
\usepackage{stmaryrd}
\usepackage[all]{xy}
\usepackage{graphicx}
\usepackage{enumerate} 
\usepackage[colorlinks=true]{hyperref} 
\usepackage[usenames,dvipsnames]{xcolor}

\usepackage{hyperref} 

\theoremstyle{plain}
\newtheorem{lem}{Lemma}[section]
\newtheorem{cor}[lem]{Corollary}

\newtheorem{thm}[lem]{Theorem}

\theoremstyle{definition}
\newtheorem{ex}[lem]{Example}
\newtheorem{rem}[lem]{Remark}
\newtheorem{dfn}[lem]{Definition}

\newcommand{\OO}{\mathcal{O}}   
\newcommand{\Ot}{\widetilde{\OO}}
\newcommand{\Ob}{\overline{\OO}}

\newcommand{\kb}{\bar k} 

\newcommand{\Ss}{\mathbf{S}}   
\newcommand{\Rr}{\mathbf{R}}  
\newcommand{\Qq}{\mathbf{Q}}   
\newcommand{\DF}{\mathbf{D}_F}   
\newcommand{\DFR}{\overline{\mathbf{D}}_F} 

\newcommand{\Z}{\mathbb{Z}}      
\newcommand{\F}{\mathbb{F}}     

\newcommand{\Ro}{\mathcal{R}}   

\newcommand{\hh}{\mathtt{h}} 
\newcommand{\hhr}{\overline{\hh}} 

\newcommand{\chh}{\mathtt{h}\text{-}CR}   
\newcommand{\chm}{\mathtt{h}\text{-}M}   

\newcommand{\End}{\operatorname{End}}  

\newcommand{\codim}{\operatorname{codim}}            

\begin{document}

\title[Motivic decompositions and Representations]{Motivic decompositions of twisted flag varieties and representations of Hecke-type algebras}

\author[A. Neshitov]{Alexander Neshitov}
\author[V. Petrov]{Victor Petrov}
\author[N. Semenov]{Nikita Semenov}
\author[K. Zainoulline]{Kirill Zainoulline}

\address{Alexander Neshitov, University of Ottawa, Canada / Steklov Institute at St.Petersburg, Russia}

\thanks{The first author was supported by the Trillium Foundation (Ontario)}

\address{Victor Petrov, Steklov Institute at St.Petersburg, Russia}

\thanks{The second author was supported by the Chebyshev
Laboratory, St. Petersburg
University via RF Government grant 11.G34.31.0026, by JSC `Gazprom Neft', and by RFBR grant 13-01-00429.}

\address{Nikita Semenov, University of Munich, Germany}

\thanks{The third author was supported by the SPP 1786 `Homotopy theory and algebraic geometry' (DFG)}

\address{Kirill Zainoulline, University of Ottawa, Canada}

\thanks{The last author was supported by the NSERC Discovery grant and the Early Researcher Award (Ontario).}

\subjclass[2010]{14F43, 14M15, 20C08, 14C15}

\keywords{linear algebraic group, torsor, flag variety, equivariant oriented cohomology, motivic decomposition, Hecke algebra}

\maketitle

\begin{abstract} 
Let $G$ be a split semisimple linear algebraic group over a field~$k_0$. Let $E$ be a $G$-torsor over a field extension $k$ of $k_0$. Let $\hh$ be an algebraic oriented cohomology theory in the sense of Levine-Morel. Consider a twisted form $E/B$ of the variety of Borel subgroups $G/B$ over $k$.

Following the Kostant-Kumar results on equivariant cohomology of flag varieties we establish an isomorphism between the Grothendieck groups of the $\hh$-motivic subcategory generated by $E/B$ 
and the category of finitely generated projective modules of certain Hecke-type algebra $H$ which
depends on the root datum of $G$, on the torsor $E$ and on the formal group law of the theory~$\hh$. 

In particular,
taking $\hh$ to be the Chow groups with finite coefficients $\F_p$ and $E$ to be a generic $G$-torsor we prove that all indecomposable submodules of an affine nil-Hecke algebra $H$ of $G$ with coefficients in $\F_p$ are isomorphic to each other and correspond to the (non-graded) generalized Rost-Voevodsky motive for $(G,p)$. \end{abstract}

\tableofcontents

\section{Introduction}

Let $G$ be a split semisimple linear algebraic group over a field $k_0$ and let $E$ be a $G$-torsor over a field extension $k$ of $k_0$. Consider a twisted form $E/B$ of the variety of Borel subgroups $G/B$ of $G$ over $k$. Observe that $E/B$ is a smooth projective variety over $k$ that in general has no rational points. For example, for $G=PGL_p$ and a non-split $E$, $E/B$ is a variety of complete flags of ideals in a central simple division algebra of a prime degree $p$ over $k$. 

Following \cite[\S64]{EKM}
consider the category of graded Chow motives  $CM(k,\F_p)$ of smooth projective varieties over $k$ with finite coefficients $\F_p$. 
According to \cite[Theorem~5.17]{PSZ} the motive $[E/B]$ of $E/B$ splits as a direct sum of Tate twists of some indecomposable motive $\Ro$, a generalization of the Rost-Voevodsky motive, i.e,
\[
[E/B]\simeq \bigoplus_{i\in I} \Ro(i).
\]
Hence, if $\langle [E/B]\rangle$ denotes a pseudo-abelian subcategory generated by the motive $[E/B]$, i.e., a minimal pseudo-abelian category containing $[E/B]$, then 
\[
\langle [E/B] \rangle=\langle \Ro(i) \rangle_{i\in I}.
\] 
Observe that in the non-graded case (in the category of motives $CM_*(k,\F_p)$ of \cite[\S64]{EKM}) all Tate twists become isomorphic and we have 
$\langle [E/B]_* \rangle=\langle \Ro_*\rangle$, where $[E/B]_*$ and $\Ro_*$ denote the respective non-graded motives.

The motive $\Ro$ has several remarkable properties (see \cite[\S5]{PSZ}).  If $p$ is not a torsion prime of~$G$, then $\Ro$ coincides with the motive of a point,
so $\langle [E/B] \rangle$ is generated by Tate twists $\F_p(i)$, $i=0..\dim G/B$. While being indecomposable over $k$,
the motive $\Ro$ becomes isomorphic to a direct sum of Tate twists over a splitting field $\kb$ of $E$ (as $\kb$ one can always take an algebraic closure of $k$ or a function field of $E/B$).
Moreover, the Poincar\'e polynomial of $\Ro$ over $\kb$
 is given by an explicit polynomial.
For example, if $G$ is an exceptional group of type $F_4$ and $p=3$, then
$\Ro|_{\kb}\simeq \F_3 \oplus \F_3(4) \oplus\F_3(8)$ for a non-split $E$.

Only very few facts are known concerning the subcategory $\langle [E/B] \rangle$ of Chow motives with integer coefficients. An integer version of the motive $\Ro$ was introduced and discussed in \cite{VZ}; in \cite{CPSZ}, \cite{CM}, \cite{SZh}  it was shown that $\langle [E/B] \rangle$ is not Krull-Schmidt (the uniqueness of a direct sum decomposition fails).

In the present paper we consider the category of  $\hh$-motives with coefficients in a commutative ring $\Rr=\hh(k)$, where $\hh$ is any algebraic oriented cohomology theory over $k$ in the sense of Levine-Morel \cite{LM}, e.g., Chow ring with integer or finite coefficients, $K$-theory,  algebraic cobordism $\Omega$ with coefficients in the Lazard ring.
Let $\langle [E/B] \rangle_\hh$ (resp. $\langle [E/B]_* \rangle_\hh$) denote its pseudo-abelian subcategory generated by the (resp. non-graded) $\hh$-motive of $E/B$. Our main result (Theorem~\ref{mainthm}) 
establishes isomorphisms between the Grothendieck groups
\begin{equation}\label{eq:maineq}
K_0(\langle [E/B] \rangle_\hh) \simeq K_0\big(\DFR^{(0)}\big)\quad \text{and}\quad K_0(\langle [E/B]_* \rangle_\hh) \simeq K_0\big(\DFR\big)
\end{equation}
of the category $\langle [E/B] \rangle_\hh$ (resp. $\langle [E/B]_* \rangle_\hh$) and the category of finitely generated projective modules over a certain $\Rr$-algebra $\DFR^{(0)}$ (resp. $\DFR$).
More precisely, the algebra $\DFR^{(0)}$ is the degree 0 component of the $\Rr$-algebra $\DFR$ defined using the formal push-pull operators (see Definition~\ref{def:ratalg}); it
depends on the root datum of $G$, on the formal group law $F$ of the theory $\hh$ and on the subring of rational cycles in $\hh(G/B)$.

If $E$ is a generic $G$-torsor, then $\DFR$ can be replaced by the formal affine Demazure algebra $\DF$.
The theory of such algebras and formal push-pull operators has been recently developed in \cite{CPZ}, \cite{HMSZ}, \cite{CZZ}, \cite{CZZ1}, \cite{CZZ2} motivated by Bernstein-Gelfand-Gelfand~\cite{BGG}, Demazure~\cite{De73}, \cite{De74}, Bressler-Evens~\cite{BE90}, \cite{BE92}, Kostant-Kumar~\cite{KK86}, \cite{KK90}, Brion \cite{Br97}, Totaro \cite{To} and Edidin-Graham \cite{EG}. The key properties of $\DF$ are 
\begin{itemize}
\item[-]
It is a free module over the $T$-equivariant oriented cohomology ring $\Ss=\hh_T(k)$ of a point, where $T$ is a split maximal torus in $G$ \cite{CZZ}.
\item[-]
Its $\Ss$-dual  $\DF^\star=Hom_\Ss(\DF,\Ss)$ is isomorphic to the $T$-equivariant oriented cohomology ring $\hh_T(G/B)$ of $G/B$ \cite{CZZ2}.
\item[-]
Its structure (generators and relations) is very close to those of the affine Hecke algebra \cite{HMSZ}.
\end{itemize}

For example, if $\hh(-) =CH(-;\F_p)$ is the Chow ring with finite coefficients, then $\DF^\star\simeq CH_T(G/B;\F_p)$ is the $T$-equivariant Chow ring and $\DF=\mathbf{H}_{nil,p}$ is the affine nil-Hecke algebra over $\F_p$ (in the notation of Ginzburg \cite[\S12]{Gi})  which is a free module of rank $|W|$ over the polynomial ring $\Ss=\F_p[x_1,\ldots,x_n]$, where $n$ is the rank of $G$ and $W$ is the Weyl group. 

For generic $E$ the isomorphisms \eqref{eq:maineq} then turn into (see Corollary~\ref{cor:nilHeckeex})
\[
K_0(\langle \Ro(i) \rangle_{i\in I}) \simeq K_0\big(\mathbf{H}_{nil,p}^{(0)}\big) \quad \text{and}\quad K_0(\langle \Ro_*\rangle) \simeq K_0\big(\mathbf{H}_{nil,p}\big),
\]
where the Tate twists $\Ro(i)$ correspond to indecomposable $\mathbf{H}_{nil,p}^{(0)}$-submodules. 
Moreover, there is a ring isomorphism 
\[
\mathbf{H}_{nil,p}\simeq Mat_{|W|/r}(\End(P_*)),
\]
where $P_*$ is the projective  $\mathbf{H}_{nil,p}$-module corresponding to $\Ro_*$ and $r$ is the $p$-part of the product of $p$-exceptional degrees of the group $G$.

The latter isomorphism specialized to $G=SL_n$ and $\hh=CH$ gives \cite[3.1.16]{Ro} and \cite[Prop. 3.5]{La}.
Indeed, in this case $E$ is split, $r=1$ and $\Ss$ is a free $\Ss^W$-module with $\hh(G/B)\simeq \Rr\otimes_{\Ss^W}\Ss$.
Then by Lemma~\ref{lem:surjc} one obtains that 
\[\mathbf{H}_{nil,p}\simeq \Ss^{W}\otimes_{\Rr}Mat_{n!}(\Rr)\simeq Mat_{n!}(\Ss^W).\]

In the paper we restrict ourselves to varieties $E/B$ of Borel subgroups only. 
However, by \cite{CPSZ} $B$ can be replaced by any special parabolic subgroup $P$ without affecting the isomorphism \eqref{eq:maineq} for non-graded motives.
For instance, for $G=PGL_{n}$, $\hh=CH(-;\Z)$ and $E$ corresponding to a generic central division algebra $A$ of degree $n$ we get
\[
K_0(\langle [SB(A)]_*\rangle)\simeq K_0(\mathbf{H}_{nil,\Z}), 
\]
where $SB(A)$ is the Severi-Brauer variety of $A$ and $\mathbf{H}_{nil,\Z}$ is the affine nil-Hecke algebra for $PGL_n$ with integer coefficients.

\medskip

The paper is organized as follows.
In section~\ref{sec:orcoh} we recall definitions and basic facts concerning  Borel-Moore homology $\hh$ and
the respective category of $\hh$-motives. We state a version of the K\"unneth isomorphism for cellular spaces.
In the next section we generalize it to the equivariant setting.
In section~\ref{sec:convo} we introduce the convolution product on the equivariant cohomology of products and study its properties. In the next section we identify the equivariant cohomology of $G$ with respect to the convolution product
with the endomorphism ring of $T$-equivariant cohomology of $G/B$ and then in section~\ref{sec:affdem} with the formal affine Demazure algebra.
In section~\ref{sec:ratalg} we introduce the notion of a rational algebra of push-pull operators $\DFR$ and identify
it with the subring of rational endomorphisms. In the last section we prove isomorphisms \eqref{eq:maineq} and provide applications and examples.

\medskip

\paragraph{\bf Acknowledgements} We are grateful to Victor Ginzburg for comments on representation theory of Hecke-type algebras.

\section{Oriented (co-)homology} \label{sec:orcoh}

We recall definitions of an algebraic oriented Borel-Moore homology and of the respective category of correspondences. We also recall a version of the K\"unneth isomorphism for cellular spaces (Lemmas~\ref{cellular} and \ref{cor1}). 

\medskip

Fix a smooth scheme $S$ over a field $k$.
Let $Sch_S$ denote the category of finite type quasi-projective separated $S$-schemes and
let $Sm_S$ denote its full subcategory
consisting of smooth quasi-projective $S$-schemes.

Following~\cite[Def.~5.1.3]{LM} consider 
an oriented graded Borel-Moore homology theory $\hh_\bullet$ 
defined on some admissible \cite[(1.1)]{LM} subcategory $\mathcal{V}$ of $Sch_S$. So that
there are pull-backs $f^*\colon \hh_\bullet(X) \to \hh_{\bullet+d}(Y)$ for l.c.i. morphisms $f\colon Y\to X$ in $\mathcal{V}$ of relative dimension $d$
and push-forwards $f_*\colon \hh_\bullet(Y) \to \hh_\bullet(X)$ for projective morphisms $f\colon X\to Y$ in $\mathcal{V}$.
According to \cite[Prop.~5.2.1]{LM} the Borel-Moore homology $\hh_\bullet$ restricted to $Sm_S$ defines an algebraic oriented cohomology theory $\hh^\bullet$ (with values in the category of graded commutative rings with unit) in the sense of \cite[Def.~1.1.2]{LM} by 
\[\hh^{\dim_S X-\bullet}(X):=\hh_\bullet(X), \quad X\in Sm_S.\]
If the (co-)dimension is clear from the context we will write simply $\hh(X)$.

\medskip

Following~\cite[\S 63]{EKM} and \cite[\S2]{VZ}  we define the category of $\hh$-correspondences $\chh(S)$ over $S$.
The objects are pairs $([X\to S],i)$, where $[X\to S]$ is an isomorphism class of a smooth projective map $X\to S$ and $i\in \Z$. 
The morphisms are defined by
\[Hom_{\chh(S)}(([Y\to S],i),([X\to S],j)):=\bigoplus_l Hom_{i-j}([Y_l\to S],[X\to S]),\]
taken over all connected components $Y_l$ of $Y$, where
\[
Hom_\bullet([Y_l\to S],[X\to S]):=\hh_{\dim_SY_l +\bullet}(Y_l \times_S X).
\]
The composition of morphisms is given by the correspondence product.
Namely, if $p_i\colon X_1\times_S X_2\times_S X_3 \to X_j\times_S X_{j'}$ denotes the projection obtained by removing the $i$-th coordinate, then given $\alpha\in \hh(X_1\times_S X_2)$ and $\beta\in \hh(X_2\times_S X_3)$
we set
\begin{equation}\label{eq:corrprod}
\beta\circ\alpha:=(p_2)_*(p_1^*(\beta)\cdot p_3^*(\alpha)) \in \hh(X_1\times_S X_3).
\end{equation}
The idempotent completion of $\chh(S)$ denoted by $\chm(S)$ is called the category of $\hh$-motives. 
We simply write $[X]$ for the respective class in $\chm(S)$.

We also consider the non-graded version of $\chh(S)$ and of $\chm(S)$ denoted by $\chh_*(S)$ and $\chm_*(S)$ respectively, were the objects are given by isomorphisms classes $[X\to S]$ of smooth projective maps and the morphisms are defined by
\[
Hom_{\chh_*(S)}([Y\to S],[X\to S]):=\hh(Y\times_S X).
\]

\begin{dfn}(cf. \cite[(CD')]{LM})\label{dfn:cell} Let $X$ be smooth projective over $S$. Suppose that there is a filtration
by proper closed subschemes 
\[
\emptyset=X_{-1}\subset X_0\subset X_1\subset\ldots\subset X_n=X\]
such that 
\begin{itemize}
\item each irreducible component $X_{ij}$ of $X_i\setminus X_{i-1}$ is a locally trivial affine fibration over $S$ of rank $d_{ij}$, and
\item the closure of $X_{ij}$ in $X$ admits a resolution of singularities $\widetilde X_{ij}\to \overline{X}_{ij}$ over~$S$;
we set $g_{ij}\colon\widetilde X_{ij}\to \overline{X}_{ij}\hookrightarrow X$ and, therefore, $(g_{ij})_*(1_{\widetilde X_{ij}})\in\hh_{d_{ij}}(X)$.
\end{itemize}
We call such $X$ (together with the filtration) a cellular space over $S$.
\end{dfn}

\begin{dfn}\label{dfn:CD}
We say that the theory $\hh$ satisfies the cellular decomposition (CD) property if given a cellular space $X$ over $S$
the respective elements $(g_{ij})_*(1_{\widetilde X_{ij}})$ form a $\hh(S)$-basis of $\hh(X)$.
\end{dfn}

\begin{ex} The property (CD) holds for any oriented Borel-Moore homology $\hh$ over a field $k$ of characteristic $0$.

Indeed, the same reasoning as in~\cite[Thm.~66.2]{EKM} shows that for every  $Z\in Sm_S$ there is an isomorphism 
\[\sum (g_{ij})_*(1)\times id_Z\colon\bigoplus_{ij}CH_{\bullet-d_{ij}}(Z)\to CH_\bullet(Z\times_SX).\]
By the Yoneda lemma (cf.~\cite[Lemma~63.9]{EKM}) the latter induces an isomorphism in the category $CM(S)$ (cf. \cite[Cor.~66.4]{EKM}).

Following~\cite[\S 2]{VY} consider the specialization functor $\Omega\text{-}M(S)\to CM(S)$, $[f\colon Y\to X]\mapsto f_*(1_Y)$. It is surjective on the classes of objects and morphisms. Moreover, for every $X$ the kernel of 
\[
\Omega_{\dim_S X}(X\times_S X)\longrightarrow CH_{\dim_SX}(X\times_S X)
\] is $\Omega_{\geqslant 1}(k)\cdot \Omega_{\bullet}(X\times_S X)$ by~\cite[Rem.4.5.6]{LM}. Hence for every $y$ in this kernel \[
y^{\circ (\dim_S X+1)} \in \Omega_{\dim_S X}(X\times_S X)\cap(\Omega_{\geqslant (dim_S X+1)}(k)\cdot \Omega_\bullet(X\times_S X)).\] So $y=0$ since $\Omega_{<0}(Y)=0$. Therefore, the kernel of \[
\End_{\Omega\text{-}M(S)}([X],i)\to \End_{CM(S)}([X],i)\] consists of nilpotents. 

Finally, by~\cite[Lemma 2.1]{VY} the isomorphism $\sum_{ij}(g_{ij})_*(1)$ in $CM(S)$ can be lifted to an isomorphism in the category $\Omega\text{-}M(S)$. Specializing it via $\Omega\to\hh$ we obtain the desired isomorphism.
\end{ex}

\begin{lem}\label{cellular} 
Assume that $\hh$ satisfies the property (CD). Let $X$ be a cellular space over $S$.
Then there is an isomorphism in $\chm(S)$
\[\sum_{ij} (g_{ij})_*(1_{\widetilde{X}_{ij}})\colon\bigoplus_{ij}([S],d_{ij})\to [X], \]
where $(g_{ij})_*(1_{\widetilde X_{ij}})\in\hh_{d_{ij}}(X)=Hom_{\chm(S)}(([S],d_{ij}),[X])$.
\end{lem}

\begin{proof}
Transversal base change implies that 
there is an isomorphism 
\[\sum (g_{ij})_*(1)\times id_Z\colon\bigoplus_{ij} \hh_{\bullet-d_{ij}}(Z\times_S S)\to \hh_\bullet(Z\times_SX)\]
for any $Z$ smooth projective over $S$. So by the Yoneda lemma (cf.~\cite[Lemma~63.9]{EKM}) it induces an isomorphism in $\chm(S)$ (cf. \cite[Cor.~66.4]{EKM}).
\end{proof}

\begin{lem}\label{cor1} Assume that $\hh$ satisfies the property (CD). Let $X$ be a cellular space over $S$.
The pairing $(\cdot,\cdot)\colon\hh(X)\otimes_{\hh(S)} \hh(X)\to\hh(S)$ given by $(a,b)=p_*(ab)$ is non-degenerate
and the map \[
f\colon(\hh(X\times_S X),\circ)\to \End_{\hh(S)}\hh(X)\quad\text{ given by }a\mapsto f_a, \; f_a(x)=(p_2)_*(p_1^*(x)\cdot a)\] is an $\hh(S)$-linear isomorphism of graded rings. In particular, it gives an $\hh(S)$-linear isomorphism
\[
(\hh_{\dim_S X}(X\times_S X),\circ)\simeq \End_{\chm(S)}(X).
\]
\end{lem}

Observe that the endomorphism ring of $\hh(S)$-linear operators $\End_{\hh(S)}(\hh(X))$ is a graded ring.
Its $n$-th graded component consists of operators increasing the codimension by $n$.
By definition the subring of degree-$0$ operators (preserving the codimension) coincides with
$\End_{\chm(S)}(X)$.

\begin{proof}
By the previous lemma there is an isomorphism 
\[\bigoplus_{ij}\hh(S)=\bigoplus_{k=-\infty}^{\infty}Hom(([S],k),\oplus_{ij}([S],d_{ij}))\stackrel{\simeq}\to\bigoplus_{k=-\infty}^{\infty}Hom(([S],k),[X])=\hh(X),\]
where each component is given by $x\mapsto x\cdot (g_{ij})_*(1)$.
Let $\sum_{ij} a_{ij}\colon [X]\to\oplus_{ij}([S],d_{ij})$ be the inverse isomorphism in $\chm(S)$. Observe that \[a_{ij} \in Hom([X],([S],d_{ij}))=\hh_{\dim(X/S)-d_{ij}}(X).\] Since $a_{ij}\circ (g_{ij})_*(1)=p_*(a_{ij}\cdot (g_{ij})_*(1))=\delta_{i,j}$, the pairing $(\cdot,\cdot)$ is non-degenerate.

The pairing $(\cdot,\cdot)$ gives an isomorphism $\hh(X)\to Hom_{\hh(S)}(\hh(X),\hh(S))$ and, hence, an isomorphism 
$\End_{\hh(S)}\hh(X)\stackrel{\simeq}\to\hh(X)\otimes_{\hh(S)}\hh(X)$. Consider the composition
\[\rho\colon\hh(X\times_SX)\stackrel{f}\to\End_{\hh(S)}\hh(X)\stackrel{\simeq}\to\hh(X)\otimes_{\hh(S)}\hh(X)\]
and a map $\pi\colon\hh(X)\otimes\hh(X)\to\hh(X\times_SX)$ given by $\pi(a\otimes b)=p_1^*(a)\cdot p_2^*(b)$.

By definition, we have \[f_{p_1^*(a)p_2^*(b)}(x)=(p_2)_*(p_1^*(x)p_1^*(a)p_2^*(b))=(x,a)b.\] Hence, $\rho(\pi(a\otimes b))=a\otimes b$ and the map $\rho$ is surjective.
By the property (CD) for $X\times_S X\to X$, $\hh(X\times_S X)$ is a free $\hh(X)$-module of rank $rk_{\hh(S)}\hh(X)$. Thus, $\rho$ is a surjective homomorphism between free modules of the same rank, hence, it is an isomorphism. 
\end{proof}

Let $C$ be any pseudo-abelian category. For an object $X\in C$ consider a subcategory $\langle X\rangle$ generated by $X$, i.e.,
the smallest pseudo-abelian subcategory of $C$ that contains $X$.

\begin{lem}
The category $\langle X\rangle$ is equivalent to the category of finetely generated projective $End_C(X)$-modules.
\end{lem}
\begin{proof}
Denote $End_C(X)$ by $R$. Every element $Y$ of $\langle X\rangle$ is isomorphic to the image $p(X^{\oplus n})$ of some idempotent $p\in End_C(X^{\oplus n})=Mat_n(R)$ and $Hom_C(X,p(X^{\oplus n}))=p(R^n)$. Note that 
\[Hom_C(p(X^{\oplus n}),p'(X^{\oplus n'}))=p'Hom_{C}(X^{\oplus n},X^{\oplus n'})p=Hom_R(p(R^n),(p'R^{n'})).\]
Then the functor $Y\mapsto Hom_C(X,Y)$ establishes an equivalence between $\langle X\rangle$ and the category of finitely generated projective right $R$-modules.
\end{proof}
\begin{cor}\label{cor:end}
The category $\langle [E/B]\rangle_{\hh}$ (resp. $\langle [E/B]_*\rangle_{\hh}$) is equivalent to the category of finitely generated projective modules over the endomorphism ring of the (resp. non-graded) $\hh$-motive of $E/B$.
\end{cor}

\section{The equivariant K\"unneth isomorphism}

In the present section we introduce an equivariant Borel-Moore homology following \cite[\S2]{CZZ} and \cite{HML}. We provide an equivariant analogue
of the K\"unneth isomorphism (Lemma~\ref{lem:eqkun}). 

\medskip

Let $G$ be a smooth group scheme over $S$. Consider an admissible subcategory $\mathcal{V}^G$ of the category of $G$-varieties $X\in Sch_S$ with $G$-equivariant morphisms. By a $G$-equivariant oriented (graded) Borel-Moore homology theory we will call an additive functor $\hh^G_\bullet$ from $\mathcal{V}^G$ to graded abelian groups such that 

\medskip

\paragraph{1} There are pull-backs for l.c.i. maps and push-forwards for projective maps that satisfy 
\begin{itemize}
\item[(TS)] (l.c.i. base change) For a Cartesian square 
$
\xymatrix{
 X'\ar[r]^{f'}\ar[d]_{g'} & Y'\ar[d]^{g}\\
X\ar[r]^{f} & Y}
$
where $f$ (hence $f'$) is l.c.i. and $g$ (hence $g'$) is projective, we have $f^*g_*=g'_*(f')^*$.
\item[(Loc)] (localization) If $U\subset X$ is an open $G$-equivariant embedding with $Z=X\setminus U$, then there is a right exact sequence:
\[\hh^G_\bullet(Z)\to\hh^G_\bullet(X)\to\hh^G_\bullet(U)\to 0.\]
\end{itemize}

\paragraph{2}
The functor $\hh_\bullet^G$ restricted to $Sm_S$ defines a graded $G$-equivariant oriented cohomology theory $\hh^\bullet_G$ in the sense of \cite{CZZ2} (we refer to \cite[\S2, A1-9]{CZZ2} for the precise definition) by
\[
 \hh^{\dim_S X-\bullet}_G(X):=\hh_\bullet^G(X),\quad X\in Sm_S.
 \]
 In addition to the axioms of \cite[\S2]{CZZ2} we require that $\hh_G$ satisfies the following stronger version of the homotopy invariance axiom:
  \begin{itemize}
  \item[(HI)] (extended homotopy invariance)  
  Let $p\colon Y\to X$ be a $G$-equivariant torsor of a vector bundle of rank $r$ over $X$, then the pull-back induced by projection
 \[p^*\colon   \hh^\bullet_G(X) \to \hh^\bullet_G(Y)\] is an isomorphism.
  \end{itemize}
If a variety is smooth we will always use the cohomology notation.

\begin{ex}\label{ex:equi}
Given a linear algebraic group $G$ over a field $k$ of characteristic zero an example of such $G$-equivariant Borel-Moore homology theory $\hh^G_\bullet$ was constructed in \cite{HML}
as follows.

Consider a system of $G$-representations $V_i$ and its open subsets $U_i\subseteq V_i$ such that
\begin{itemize}
\item $G$ acts freely on $U_i$ and the quotient $U_i/G$ exists as a scheme over $k$,
\item $V_{i+1}=V_i\oplus W_i$ for some representation $W_i$,
\item $U_i\subseteq U_i\oplus W_i\subseteq U_{i+1}$, and $U_i\oplus W_i\to U_{i+1}$ is an open inclusion, and
\item  $\codim(V_i\setminus U_i)$ strictly increases.
\end{itemize}
Such a system is called a good system of representations of $G$. 

Let $X\in Sch_k$ be a $G$-variety. Following~\cite[\S3 and \S5]{HML} the inverse limit induced by pull-backs 
\[
\varprojlim_i \hh_{\bullet-\dim G+\dim U_i}(X\times^G U_i),\quad X\times^G U_i=(X\times_k U_i)/G,
\] 
does not depend 
on the choice of the system $(V_i,U_i)$ and, hence, defines the $G$-equivariant oriented homology group $\hh^G_\bullet(X)$.
\end{ex}

In the present paper we will extensively use the following property (cf. \cite[\S2, A6]{CZZ2})
of an equivariant theory
\begin{itemize}
\item[(Tor)]
Let $X\to X/G$ be a $G$-torsor over $S$ and a $G'$-equivariant map for some group scheme $G'$ over $S$. Then
there is an isomorphism
\[
\hh_{G\times G'}^\bullet(X)\stackrel{\simeq}\longrightarrow \hh_{G'}^\bullet(X/G).
\]
that is natural with respect to the maps of pairs
\[(\phi,\gamma)\colon(X,G\times G')\to (X_1,G_1\times G_1'),
\quad \phi(x\cdot (g,g'))=\phi(x)\cdot\gamma(g,g').
\]
\end{itemize}
Observe that the theory of Example~\ref{ex:equi} satisfies this property 
by~\cite[Prop.~27]{HML}.

\medskip

We have the following equivariant analogues of Definitions~\ref{dfn:cell} and~\ref{dfn:CD}

\begin{dfn} Let $X\in \mathcal{V}^G$. Suppose that there is a filtration
by $G$-equivariant proper closed subschemes 
\[
\emptyset=X_{-1}\subset X_0\subset X_1\subset\ldots\subset X_n=X\]
such that 
\begin{itemize}
\item each irreducible component $X_{ij}$ of $X_i\setminus X_{i-1}$ is a $G$-equivariant (locally trivial) affine fibration over $S$ of rank $d_{ij}$, and
\item the closure of $X_{ij}$ in $X$ admits a $G$-equivariant resolution of singularities $g_{ij}\colon\widetilde X_{ij}\to \overline{X}_{ij}$ over~$S$.
\end{itemize}
We call such $X$ (together with the filtration) a $G$-equivariant cellular space over $S$.
\end{dfn}

\begin{dfn}
We say that the equivariant theory $\hh^G$ satisfies the cellular decomposition (CD) property if given a $G$-equivariant cellular space $X$ over $S$
the respective elements $(g_{ij})_*(1_{\widetilde X_{ij}})$ form a $\hh^G(S)$-basis of $\hh^G(X)$.
\end{dfn}

\begin{lem}\label{commute}
Suppose a morphism $f\colon X\to Y$ in $Sm_k$ factors as $f\colon X\stackrel{z}\to L\stackrel{j}\to Y$ where $p\colon L\to X$ is a vector bundle, $z\colon X\to L$ is a zero section and $j$ is an open embedding. 

Then for every projective map $a\colon Y'\to Y$ and $X'=X\times_Y Y'$ the following diagram of pull-back and push-forward maps commutes (we omit the grading)
\[
\xymatrix{
\hh(X')\ar[r]^{a'_*} & \hh(X)\\
\hh(Y')\ar[u]^{f'^*}\ar[r]^{a_*} & \hh(Y)\ar[u]_{f^*}
}
\]
\end{lem}
\begin{proof}
Observe that the map $f'\colon X'\to Y'$ factors as $X'\stackrel{z'}\to L\times_YY'\stackrel{j'}\to Y'$ where $z'$ is the zero section of the vector bundle $p'\colon L'=L\times_YY'\to X'$ and $j'$ is an open embedding. Let $b$ denote the canonical map $L'\to L$. Since $j$ and $j'$ are flat, we have $j^*a_*=b_*j'^*$ by the l.c.i. base change for oriented theories. Note that by the homotopy invariance $z^*=(p^*)^{-1}$ and $z'^*=(p'^*)^{-1}$. Since $p$ and $p'$ are flat, $p^*a'_*=b_*p'^*$. Then $z^*b_*=a'_*z'^*$ and 
\[
f^*a_*=z^*j^*a_*=z^*b_*j'^*=a'_*z'^*j'^*=a'_*f'^*.\qedhere 
\]
\end{proof}
\begin{rem}
If $(V_i,U_i)$ is a good system of representations of Example~\ref{ex:equi}, then for any $G$-variety $X$ the connecting maps $X\times^G U_i\to X\times^G U_{i+1}$ factor as in~Lemma~\ref{commute}, i.e., we have $X\times^G U_i\to X\times^G (U_i\oplus W_i)\to X\times^G U_{i+1}$.
\end{rem}

\begin{ex}\label{eqcellular}
Let $\hh^G$ be the equivariant theory of Example~\ref{ex:equi}. Then the property (CD) holds for $\hh^G$.

Indeed,
consider a good system of representations $\{(V_j,U_j)\}_j$ for $X$. The subvarieties $X_i\times^GU_j$, $i=0\ldots n$ form a cellular filtration on $X\times^GU_j$ over $S\times^GU_j$. Note that $\widetilde X_i \times^GU_j$ is a resolution of singularities of $X_i\times^GU_j$. By~(CD) for $\hh$ the set $\{(f_i\times^Gid_{U_j})_*(1)\}_i$ forms a basis of $\hh(X\times^GU_j)$ as a $\hh(S\times^GU_j)$-module. By~Lemma~\ref{commute} the following diagram commutes:
\[
\xymatrix{
\hh(\widetilde X_i\times^GU_{j+1})\ar[rr]^{(g_{i,j+1})_*}\ar[d]^{\widetilde i_j^*} && \hh(X\times^GU_{j+1})\ar[d]^{i_j^*}\\
\hh(\widetilde X_i \times^GU_{j})\ar[rr]^{(g_{i,j})_*} && \hh(X\times^GU_{j})
}
\]
So
$i_m^*((f_i\times^Gid_{U_{j+1}})_*(1))=(f_i\times^Gid_{U_j})_*(1)$, which implies that
the elements $f_{i*}(1)=\lim_j((f_i\times^Gid_{U_j})_*(1))$ form a basis of $\hh^G(X)$ over $\hh^G(S)$. 
\end{ex}

As for usual oriented theories we then obtain

\begin{lem}\label{lem:eqkun} Assume that $\hh^G$ satisfies the property (CD). Let $X$ be a $G$-equivariant cellular space over $S$.
Then the pairing $(\cdot,\cdot)\colon\hh^G(X)\otimes_{\hh^G(S)} \hh^G(X)\to\hh^G(S)$ given by $(a,b)=p_*(ab)$ is non-degenerate
and the map \[
f\colon(\hh^G(X\times_S X),\circ)\to \End_{\hh^G(S)}\hh^G(X)\quad\text{ given by }a\mapsto f_a, \; f_a(x)=(p_2)_*(p_1^*(x)\cdot a)\] is an $\hh^G(S)$-linear isomorphism of rings. In particular, there is an $\hh^G(S)$-linear isomorphism
\[
(\hh^G_{\dim_S X}(X\times_S X),\circ)\to \End_{\hh^G\text{-}M(S)}(\hh^G(X)),
\]
where $\hh^G\text{-}M(S)$ is the respective category of $G$-equivariant motives.
\end{lem}

\section{The convolution product}\label{sec:convo}

In the present section we introduce the convolution product on the equivariant Borel-Moore homology (Definition~\ref{def:conv}) of the product $G\times G\times\ldots \times G$. We relate this product to the usual correspondence product for the associated torsors (Lemma~\ref{thm:surjphi}) and study its behaviour under the base change (diagram~\eqref{diag:basech}).

\medskip

Let $G$ be a smooth algebraic group over $k$ and let $E$ be a $G$-torsor
over $k$ ($G$ acts on the right).
By definition there is an isomorphism $\rho\colon E\times_k G \stackrel{\simeq}\to E\times_k E$
given on points by $(e,g)\mapsto (e,eg)$. For each $i\ge 0$ it induces an isomorphism
\[
\rho_i\colon E\times_k G^{i} \longrightarrow E^{i+1}, \quad (e,g_1,g_2,\ldots,g_i)\mapsto (e,eg_1,eg_2,\ldots, eg_i).
\]
Consider the composition
\[
\gamma_i\colon E^{i+1}\stackrel{\rho_i^{-1}}\longrightarrow E\times_k G^{i} = E\times_k G^i \stackrel{pr}\longrightarrow G^i.
\]
The coordinate-wise right $G^{i+1}$-action on $E^{i+1}$ induces an action on $E\times_k G^{i}$ and, hence, on $G^{i}$. For instance, on points it is given by
\begin{equation}\label{eq:acti}
(e,g_1,\ldots,g_i)\cdot (h_1,\ldots,h_{i+1})=(eh_1,h_1^{-1}g_1h_2,\ldots,h_1^{-1}g_ih_{i+1}).
\end{equation}

Consider projections $p_j\colon E^{i+1} \to E^i$ obtained by removing the $j$-th coordinate
and the respective $G^i$-action on $E^i$. For each $i\ge 1$, $1\le j\le i+1$ there is a commutative diagram of $G^i$-equivariant maps
\begin{equation}\label{diag:equiv}
\xymatrix{
E^{i+1} \ar[r]^{\gamma_i} \ar[d]_{p_j} & G^i \ar[d]^{\pi_j} \\
E^i \ar[r]^{\gamma_{i-1}} & G^{i-1}
}
\end{equation}
where $\pi_1(g_1,\ldots,g_i)=(g_1^{-1}g_2,\ldots,g_1^{-1}g_i)$ and $\pi_{j}(g_1,\ldots,g_i)=(g_1,\ldots,\hat g_{j-1},\ldots,g_i)$ for $j>1$.

\begin{ex} For $i=1$ it gives a commutative diagram of $G$-equivariant maps
\[
\xymatrix{
E\times_k E \ar[r]^-{\gamma_1} \ar[d]_{p_j} & G \ar[d]^{\pi_j} \\
E \ar[r]^-{\gamma_{0}} & Spec\; k
}
\]
where $\gamma_0, \pi_1, \pi_2$ are the structure maps, $p_1,p_2$ are the corresponding projections and $\gamma_1(e,eg)=g$. Moreover, if $E$ is trivial, then $\gamma_1=\pi_1\colon G\times_k G\to G$, $(g_1,g_2)\mapsto g_1^{-1}g_2$.
\end{ex}

Let $H$ be an algebraic subgroup of $G$ such that $G/H$ is a smooth variety over~$k$.
We can view $G^i$ as an $H$-torsor over $G^i/H$, where $H$ acts on $G^i$ via the $j$th coordinate of $G^{i+1}$.
By definition, the $H^i$-equivariant map $\pi_j$ factors as
\[
\pi_j\colon G^i \stackrel{q}\longrightarrow G^i/H \stackrel{\bar\pi_j}\longrightarrow G^{i-1},
\]
where the second map $\bar\pi_j$ is a fibration with a fibre $G/H$.

\begin{ex}\label{ex:lemdiag}
The map $\pi_1$ factors through the quotient maps modulo the diagonal action
\[
\pi_1\colon G^i \stackrel{q}\longrightarrow G^i/\Delta(H) \stackrel{\bar{\pi}_1}\longrightarrow G^i/\Delta(G)=G^{i-1}.
\]
which are equivariant with respect to the usual coordinate-wise $H^i$-action.
\end{ex}

Consider an equivariant Borel-Moore homology theory $\hh$. For every $1\leqslant j\leqslant i+1$ consider the action of the $j$-th copy of $H$ on $G^i$.
The property (Tor) gives an isomorphism
\begin{equation}\label{eq:red}
\hh_{H^i}(G^i/H) \stackrel{\simeq}\longrightarrow \hh_{H^{i+1}}(G^i),
\end{equation}
where $H^{i+1}$ acts on $G^i$ as in \eqref{eq:acti}. Unless explicitly mentioned we will always identify these two rings.

Set $\Ss=\hh_H(G^0)=\hh_H(k)$ and set the convolution product on $\Ss$ to be the usual intersection product.

\begin{dfn}\label{def:conv} Assume that $G/H$ is a smooth projective variety over $k$.
We define the $\Ss$-linear convolution product $'\circ'$ on $\hh_{H^{i}}(G^{i-1})$, $i\ge 2$ to be the composite 
\[\hh_{H^{i}}(G^{i-1})\otimes \hh_{H^{i}}(G^{i-1})\stackrel{\bar\pi_{i-1}^*\otimes \bar\pi_{i+1}^*}\longrightarrow \hh_{H^{i+1}}(G^{i})\otimes \hh_{H^{i+1}}(G^{i})\stackrel{'\cdot'}\longrightarrow 
\]
\[\hh_{H^{i+1}}(G^{i}) \stackrel{(\bar\pi_{i})_*}\longrightarrow \hh_{H^{i}}(G^{i-1}),\]
where $\hh_{H^{i+1}}(G^{i})$ is identified with $\hh_{H^{i}}(G^{i}/H)$ via \eqref{eq:red} and $\bar\pi_i$ is projective because so is $G/H$.

\end{dfn}

The central object of the present paper is the convolution ring $(\hh_{H^2}(G),\circ)$, i.e., the case $i=2$.
In the next sections we will show that $(\hh_{B^2}(G),\circ)$ (where $B$ is a Borel subgroup of a semisimple split $G$) can be identified with the formal affine Demazure algebra.

\begin{ex} In the case $i=3$ the convolution ring $(\hh_{H^3}(G^2),\circ)$ is isomorphic to $\hh_{\Delta(H)}((G/H)^2)$
with respect to the usual correspondence product. 
Indeed, the maps $\pi_i\colon G^3 \to G^2$, $i=2,3,4$ induce $\Delta(H)$-equivariant projections $(G/H)^3\to (G/H)^2$.
The isomorphism then follows by (Tor).

Observe that if $G/H$ is an $H$-equivariant cellular space
and $\hh_H$ satisfies (CD), then by Lemma~\ref{lem:eqkun} there is an $\Ss$-linear ring isomorphism
\[
(\hh_{H^3}(G^2),\circ)\simeq \End_{\Ss} \hh_H(G/H).
\]
\end{ex}

\begin{lem}\label{lem:convprod} For $i\ge 1$ the map $\pi_1$ induces an injective ring homomorphism with respect to the convolution products
\[
(\hh_{H^{i}}(G^{i-1}),\circ) \stackrel{\bar\pi_1^*}\longrightarrow (\hh_{H^{i+1}}(G^{i}),\circ).
\]
\end{lem}

\begin{proof}
For $i=1$ it follows from the fact that the convolution product on $\hh_{H^2}(G)$ is $\Ss$-linear.

For $i\ge 2$ for each $i-1\le j\le i+1$ we have $\pi_{j}\circ\pi_1=\pi_1\circ\pi_{j+1}$. Since push-forwards commute with flat pull-backs by (TS), there are commutative diagrams in equivariant cohomology
\[
\xymatrix{
\hh_{H^{i+1}}(G^{i}) \ar[r]^{\bar\pi_1^*} \ar[d]_{(\bar\pi_{i})_*}& \hh_{H^{i+2}}(G^{i+1}) \ar[d]_{(\bar\pi_{i+1})_*}\\
\hh_{H^{i}}(G^{i-1}) \ar[r]^{\bar\pi_1^*} \ar@<-5pt>[u]_{\bar\pi_{i-1}^*,\bar\pi_{i+1}^*} & \hh_{H^{i+1}}(G^{i})  \ar@<-5pt>[u]_{\bar\pi_{i}^*,\bar\pi_{i+2}^*} 
}
\]
Finally, there is a $H^i$-equivariant section of the map $\bar\pi_1\colon G^{i}/\Delta(H) \to G^{i-1}$ given by $(g_1,\ldots,g_{i-1})\mapsto (1,g_1,\ldots,g_{i-1})$,
so $\bar\pi_1^*$ is injective.
\end{proof}

\begin{lem}\label{thm:surjphi} 
The map $\gamma_1$ induces a ring homomorphism 
\[
(\hh_{H^2}(G),\circ)\stackrel{\gamma_1^*}\longrightarrow (\hh_{H^2}(E^2),\circ)\stackrel{\simeq}\longrightarrow (\hh((E/H)^2),\circ),
\] 
where the last ring is viewed with respect to the correspondence product \eqref{eq:corrprod}.
\end{lem}

\begin{proof}
By (TS) the diagram \eqref{diag:equiv} gives rise to commutative diagrams in cohomology \[
\xymatrix{
\hh_{H^3}(G^2) \ar[r]^{\gamma_2^*} \ar[d]_{(\bar\pi_2)_*}& \hh_{H^3}(E^3) \ar[d]_{(p_2)_*}\\
\hh_{H^2}(G) \ar[r]^{\gamma_1^*} \ar@<-10pt>[u]_{\bar\pi_{1}^*,\bar\pi_{3}^*} &   \hh_{H^2}(E^2) \ar@<-10pt>[u]_{p_{1}^*,p_{3}^*} 
}
\]
The last isomorphism follows by (Tor).
\end{proof}

Let $\kb$ denote the splitting field of a $G$-torsor $E$ so that $G_{\kb} = E_{\kb}$. 
Since the base change preserves the convolution product, combining Lemmas~\ref{lem:convprod} and~\ref{thm:surjphi} we obtain
two commutative diagrams of convolution (correspondence) rings
\[
\xymatrix{
\gamma_1^*\colon \hh_{H^2}(G) \ar[r]^-{pr^*} \ar[d]_{res_{\kb/k}} & \hh_{H^2}(E\times_k G) \ar[r]^{\rho_1^*}_{\simeq}& \hh_{H^2}(E^2) \ar[d]^{res_{\kb/k}}\\
\gamma_1^*\colon\hh_{H^2}(G_{\kb}) \ar[r]^-{\bar\pi_1^*} & \hh_{H^2}(G_{\kb}^2/\Delta(H)) \ar[r]^{q^*} & \hh_{H^2}(G_{\kb}^2)
}
\]
and
\[
\xymatrix{
\gamma_0^*\colon \hh_{H}(k) \ar[rr]^-{\rho_0^*\circ pr^*} \ar[d]_{res_{\kb/k}}^{\simeq} & & \hh_{H}(E)  \ar[d]^{res_{\kb/k}}  \\
\gamma_0^*\colon \hh_{H}(\kb) \ar[rr]^-{q^*\circ \bar\pi_1^*} & & \hh_{H}(G_{\kb})  
}
\]
where $res_{\kb/k}$ is the base change map. Combining these two diagrams we obtain a commutative diagram of convolution rings
\begin{equation}\label{diag:basech}
\xymatrix{
 \hh_H(E) \otimes_{\Ss} \hh_{H^2}(G) \ar[r]^-{(p_1^*,\gamma_1^*)} \ar[d]_{res_{\kb/k}} & \hh_{H^2}(E^2) \ar[d]^{res_{\kb/k}}\\
 \hh_H(G_{\kb}) \otimes_{\Ss} \hh_{H^2}(G_{\kb}) \ar[r]^-{(p_1^*,\gamma_1^*)} & \hh_{H^2}(G_{\kb}^2),
}
\end{equation}
where the left convolution rings are $\hh_H(E)$- and $\hh_{H}(G_{\kb})$-linear.

\section{The  subring of push-pull operators}

In the present section we prove that if $H$ is the Borel subgroup of a split semisimple linear algebraic group, then the convolution ring $\hh_{H^2}(G)$ of Definition~\ref{def:conv} can be identified with the subring of push-pull operators (Corollary~\ref{thm:endpres}). Our arguments are essentially based on the Bruhat decomposition of $G$ stated using the $G$-orbits on the product $G/H\times_k G/H$ and the resolution of singularities \eqref{eq:ressing}.

\medskip

As before assume that $G/H$ is a smooth projective variety over $k$.
In the notation of the previous section consider the $H^2$-equivariant maps of Example~\ref{ex:lemdiag}. 
\[
\pi_1\colon G^2\stackrel{q}\longrightarrow G^2/\Delta(H)\stackrel{\bar{\pi}_1}\longrightarrow G^2/\Delta(G)=G,\quad (g_1,g_2)\mapsto g_1^{-1}g_2.
\] 
Since $G^2$ is a $\Delta(G)$-torsor over $G$ ($\Delta(H)$-torsor over $G^2/\Delta(H)$), 
by the property (Tor) the induced $\Delta(G)\times H^2$-equivariant pull-backs on cohomology coincide with the forgetful maps
\begin{equation}\label{diag:comconst}
\xymatrix{
\gamma_1^*\colon\hh_{H^2}(G)\simeq\hh_{\Delta(G)\times H^2}(G^2) \ar@{^(->}[r]^-{\bar \pi_1^*} \ar[d]_\simeq& \hh_{\Delta(H)\times H^2}(G^2) \ar[r]^-{q^*} \ar[d]^\simeq & \hh_{H^2}(G^2) \ar[d]^\simeq\\
\hh_G((G/H)^2) \ar@{^(->}[r] & \hh_H((G/H)^2) \ar[r] & \hh((G/H)^2)
}
\end{equation}
Moreover, by Lemma~\ref{lem:convprod} it is a commutative diagram of convolution rings.

\medskip

Let $G$ be a split semisimple linear algebraic group over $k$ and let $\hh$ be an equivariant theory that satisfies property (CD). 
We fix a Borel subgroup $B$ of $G$ containing a split maximal torus $T$. 
By Bruhat decomposition (e.g. \cite{Sp})
\[
G=\amalg_{w\in W}B\dot{w} B,\quad \dot{w}\in N_T,
\] 
is the disjoint union of $B^2$-orbits of $G$, where $W=N_T/T$ is the Weyl group and $N_T$ is the normalizator of $T$ in $G$. Projecting this decomposition onto $X=G/B$ gives a $B$-equivariant cellular filtration on $X$ by closures $\overline{X}_w$ of affine spaces $X_w=B\dot{w} B/B$ of dimension $l(w)$ (the length of $w$).
The preimage $\pi_1^{-1}(B\dot{w} B)$ is a $\Delta(G)$-orbit in $G^2$ (here $H=B$). Let $\OO_w$ denote its image via $G^2\to X^2$ and let $\Ob_w$ denote its closure. Observe that both $\OO_w$ and $\Ob_w$ are $\Delta(G)$-invariant in $X^2$.

\medskip

By properties of the Bruhat decomposition (see \cite[\S1]{Sp}) it follows that the projection  $\OO_w \to X^2\to X$ is a torsor of a vector bundle over $X$ with fibre $X_w$. Indeed, the transition functions are affine since they are given by the action of $B$ on the left on $B\dot{w}B/B$ that is by $T$ acting on the product of the respective root subgroups 
$\prod_{\alpha\in \Phi^+\cap w(\Phi^-)} U_\alpha$ via the conjugation and, hence, by $T$ acting on the product of the respective $\mathbb{G}_a$'s via the multiplication $t\cdot x=\alpha(t)x$, $t\in T$, $x\in \mathbb{G}_a$.
So $X^2$ is a $G$-equivariant ($G$ acts diagonally) cellular space over $X$ with filtration given by the closures $\Ob_w$.

\medskip

Assume that for each $w\in W$ we are given a $G$-equivariant resolution of singularities $\Ot_w\to \Ob_w$. Let $[\Ot_w]_G$ denote the respective class in $\hh_G^{\dim_k X-l(w)}(X^2)$. Then by the property (CD) 
the cohomology $\hh_G(X^2)$ (resp. $\hh_B(X^2)$ and $\hh(X^2)$) is a free module over $\hh_G(X)$ (resp. over $\hh_B(X)$ and $\hh(X)$) with basis $\{[\Ot_w]_G\}_{w\in W}$ (resp. $\{[\Ot_w]_B\}_{w\in W}$ and $\{[\Ot_w]\}_{w\in W}$). Hence,
the forgetful maps of \eqref{diag:comconst}
send $[\Ot_w]_G\mapsto [\Ot_w]_B\mapsto [\Ot_w]$ and change the coefficients by $-\otimes_{\hh_G(X)}\hh_B(X)$ and $-\otimes_{\hh_B(X)}\hh(X)$ respectively, where the map $\Ss=\hh_G(X)\hookrightarrow \hh_B(X) \to \hh(X)$ is the classical characteristic map.

\medskip

We now construct such $G$-equivariant resolutions as follows.
For the $i$-th simple reflection $s_{i}$ we denote $X_{s_{i}}$ (resp. $\OO_{s_i}$) simply by $X_i$ (resp. by $\OO_i$).
Let $P_i$ be the minimal parabolic subgroup corresponding to a simple root $\alpha_i$ and let $q_i\colon X\to G/P_i$ denote the respective quotient map.

\begin{lem}\label{orbitasprod}
We have $\Ob_i=X\times_{G/P_i}X$ and, in particular, $\Ob_i$ is smooth.
\end{lem}

\begin{proof}
We have $(g_1B,g_2B)\in X\times_{G/P_i}X$, $g_1,g_2\in G$ if and only if $g_1P_i=g_2P_i$, so $g_2=g_1h$ for some $h\in P_i$. Since $P_i=B\cup Bs_iB$, it means that either $g_2B=g_1B$ or $g_2B=g_1Bs_iB$, so $(g_1B,g_2B)\in \OO_{s_i}\cup\Delta_X=\Ob_i$.
\end{proof}

For any $w\in W$ we choose a reduced decomposition $w=s_{i_1}s_{i_2}\ldots s_{i_l}$ and set $I_w=(i_1,i_2,\ldots,i_l)$. Consider a variety
\begin{equation}\label{eq:ressing}
\Ot_{I_w}=X\times_{G/P_{i_1}}X\times_{G/P_{i_2}}\ldots\times_{G/P_{i_l}} X.
\end{equation}
The projection on the first and the last factor $pr\colon \Ot_{I_w}\to X\times_k X$ gives a $G$-equivariant resolution of singularities of $\Ob_w$. 

\begin{thm}\label{lem:imkun} For $H=B$ or $1$, the image of $[\Ot_{I_w}]_H \in \hh_H(X\times_k X)$ under the K\"unneth isomorphism 
\[
(\hh_H(X\times_k X),\circ)\stackrel{\simeq}\longrightarrow \End_{\hh_H(k)}(\hh_H(X))
\] 
is the composition of push-pull operators $q_{i_1}^*q_{i_1*}\circ\ldots\circ q_{i_l}^*q_{i_l*}$.
\end{thm}

\begin{proof}
By definition the image of $[\Ot_{I_w}]_H$ is the $\hh_H(k)$-linear operator
\[
\hh_H^\bullet(X)\stackrel{p_1^*}\longrightarrow \hh_H(X\times_k X)\stackrel{\cdot[\Ot_{I_w}]}\longrightarrow \hh_H(X\times_k X)\stackrel{p_{2*}}\longrightarrow \hh_H^{\bullet-l(w)}(X).
\]
By the projection formula and (TS) it can be also written as
\[
\hh_H^\bullet(X)\stackrel{pr_{l+1}^*}\longrightarrow\hh_H(\Ot_{I_w})\stackrel{pr_1*}\longrightarrow\hh_H^{\bullet-l(w)}(X),
\]
where $pr_j$ denotes the projection on the $j$-th coordinate (recall that $p_j$ denotes the projection obtained by removing the $j$-th coordinate).

By the property (TS) we obtain a commutative diagram
\[
\xymatrix{
\hh_H(X)\ar[r]^{pr_{2}^*} \ar[d]_{q_{i_l*}} & 
\hh_H(\Ot_{i_{l}}) \ar[d]_{pr_{1*}} \ar[r]^-{pr_{23}^*} & 
\hh_H(\Ot_{(i_{l-1},i_l)}) \ar[d]_{pr_{12*}} \ar[r]^-{pr_{234}^*} &
\ldots \ar[r] &
\hh_H(\Ot_{I_w}) \ar[d] \\
\hh_H(G/P_{i_l}) \ar[r]^{q_{i_l}^*} & 
\hh_H(X)\ar[d]_{q_{i_{l-1}*}} \ar[r]^{pr_{2}^*} & 
\hh_H(\Ot_{i_{l-1}}) \ar[d]_{pr_{1*}} & 
& \ldots \ar[d] \\
& 
\hh_H(G/P_{i_{l-1}}) \ar[r]^{q_{i_{l-1}}^*} & 
\hh_H(X) \ar[d]_{q_{i_{l-2}*}} & 
&
\ldots \ar[d] \\
& 
&
\ldots \ar[r] &
\ldots \ar[r] & 
\hh_H(X)
}
\]
where $pr_{ijk\ldots}$ denote the projection on the $i$-th, $j$-th, $k$-th, $\ldots$, coordinates. The result then follows since the top horizontal row gives $pr_{l+1}^*$ and the right vertical column gives $pr_{1*}$. 
\end{proof}
Observe that the theorem can not be stated for $H=G$ as $X$ is not a $G$-equivariant cellular space so we can not use the K\"unneth isomorphism of Lemma~\ref{lem:eqkun}.

Combining Diagram \eqref{diag:comconst} and Theorem~\ref{lem:imkun} we obtain

\begin{cor}\label{thm:endpres} 
There is a commutative diagram of convolution rings
\[
\xymatrix{
\hh_{B^2}(G)\ar@{^{(}->}[r]^-{\bar\pi_1^*}  & \hh_{\Delta(B)\times B^2}(G^2) \ar[r]^-\simeq  \ar[d]_{q^*}& \hh_B(X^2) \ar[r]^-\simeq \ar[d] & End_{\Ss}(\hh_B(X)) \ar[d] \\
 & \hh_{B^2}(G^2) \ar[r]^-\simeq & \hh(X^2) \ar[r]^-\simeq & End_{\Rr}(\hh(X))
}
\]
where the image of $(\hh_{B^2}(G),\circ)$ in $End_{\Ss}(\hh_B(X))$ is the subring generated by the push-pull operators 
$q_i^*q_{i_*}$ (of degree $(-1)$) and the image of the forgetful map $\Ss=\hh_G^\bullet(X)\to\hh_B^{\bullet}(X)$ (of degrees '$\bullet$') and the last vertical arrow is induced by the augmentation map $\Ss \to \Rr=\hh(k)$.
\end{cor}

\section{Self-duality of the algebra of push-pull operators}\label{sec:affdem}

In the present section
we identify the convolution ring $\hh_{B^2}(G)$ with the formal affine Demazure algebra $\DF$ of \cite{HMSZ} and show that it is self-dual with respect to the convolution product (Theorem~\ref{thm:Demazure}). 
Our arguments are based on the results of~\cite{HMSZ}, \cite{CZZ}, \cite{CZZ1} and, especially, \cite{CZZ2}. We use the notation of \cite{CZZ2}.

\medskip

Recall that algebraic oriented cohomology theories $\hh$ correspond (up to universality) to one-dimensional commutative formal group laws $F(u,v)$: the formal group law corresponds to $\hh$ by means of the Quillen formula expressing the first characteristic classes 
\[
c_1^\hh(\mathcal{L}_1\otimes \mathcal{L}_2)=F(c_1^\hh(\mathcal{L}_1),c_1^\hh(\mathcal{L}_2))
\]
and
the respective cohomology theory $\hh$ is defined from $F$ by tensoring with the algebraic cobordism 
\[
\hh(-)=\Omega(-)\otimes_{\Omega(k)}\Rr,
\]
where $\Omega(k)\to \Rr$ defines $F$ by specializing the coefficients in the Lazard ring
(see \cite[\S2]{CZZ2} for details).
For example, the additive formal group law correspond to Chow groups and the periodic multiplicative law corresponds to $K$-theory.

By~\cite[Thm.~3.3]{CZZ2} the completed $B$-equivariant coefficient ring $\Ss=\hh_B(k)$ can be identified with the formal group algebra $\Rr[[T^*]]_F$, where $T^*$ is the group of characters of a split maximal torus $T\subset B$ and $F$ is the respective formal group law.

Following \cite[\S5]{CZZ2} (we assume that $\Ss$ satisfies regularity condition \cite[5.1]{CZZ2}) consider the localized algebra $\Qq=\Rr[[T^*]]_F[\frac{1}{x_{\alpha}}]_{\alpha}$ (where $\alpha$ runs through all simple roots) and the smash  products $\Qq_W=\Qq\otimes_\Rr \Rr[W]$ and $\Ss_W=\Ss\otimes_\Rr \Rr[W]$ with the multiplication given by 
\[
q\delta_w\cdot q'\delta_{w'}=q(wq')\delta_{ww'}\] for $q,q'\in \Qq$ (respectively $\Ss$) and $w,w'\in W$ (the Weyl group). 
Consider the duals $\Qq_W^\star=Hom_\Qq(\Qq_W,\Qq)$ and $\Ss_W^\star=Hom_\Ss(\Ss_W,\Ss)$.
By definition $\Qq_W^\star$ and $\Ss_W^\star$ can be identified with the ring of functions
$Hom(W,\Qq)$ and $Hom(W,\Ss)$ respectively

As in \cite[Def.~6.2, 6.3]{HMSZ}
for each simple root $\alpha_i$ of the root system for $G$
define the push-pull element \[
Y_i=(1+\delta_i)\frac{1}{x_{-i}}\in \Qq_W.
\]
Define the formal affine Demazure algebra $\DF$ as the subalgebra of $\Qq_W$ generated by multiplications by $\Ss$ and the elements $Y_i$.

By  \cite[Thm.~7.9]{CZZ} (see also \cite[Thm.~5.14]{HMSZ}) the $\Rr$-algebra $\DF$ satisfies the following (complete) set of relations: for $i,j=1\ldots rk(G)$ and $u\in \Ss$
\begin{itemize}
\item $Y_i^2=\kappa_i Y_i$, where $\kappa_i=\tfrac{1}{x_i}+\tfrac{1}{x_{-i}}$ and $x_i=x_{\alpha_i}$,
\item $Y_iu=s_i(u)Y_i+\Delta_{-i}(u)$, where $\Delta_{-i}(u)=\tfrac{u-s_i(u)}{x_{-i}}$, 
\item $(Y_iY_j)^{m_{ij}}-(Y_jY_i)^{m_{ij}}=\sum_{I_w} c_{I_w} Y_{I_w}$, where the sum
is taken over all reduced expressions $I_w$ of elements $w$ of the subgroup $\langle s_i,s_j\rangle \subseteq W$, and the coefficients $c_{I_w}$ are given by the formulas of \cite[Prop.~5.8]{HMSZ}

\end{itemize}

\begin{ex} If $F$ corresponds to Chow groups, then $\DF=\mathbf{H}_{nil}$ is the affine nil-Hecke algebra over $\Z$ in the notation of \cite{Gi}.
If $F$ corresponds to $K$-theory, then $\DF$ is the 0-affine Hecke algebra over $\Z$ ($q\to 0$ in the affine Hecke algebra).
If $F$ corresponds to the generic hyperbolic formal group law of \cite[\S9]{CZZ1}, then by \cite[Prop.~9.2]{CZZ1} the constant part of $\DF$ is isomorphic
to the localized classical Iwahori-Hecke algebra.
\end{ex}

Let $\DF^\star=Hom_{\Ss}(\DF,\Ss)$ denote its dual.
Observe that the main result of \cite{CZZ2} (Thm.~8.2 loc.cit.) says that $\DF^\star$ is isomorphic to the $\Rr$-algebra $\hh_{T}(X)$.
We then obtain the following generalization of \cite[Prop.~12.8]{Gi}

\begin{thm}\label{thm:Demazure} Let $G$ be a split semisimple linear algebraic group over a field $k$ and let $\hh$ be an equivariant theory
that satisfies property (CD).

Then the convolution algebra $(\hh_{B^2}(G),\circ)$ is isomorphic (as an $\Rr$-algebra) to the formal affine Demazure algebra $\DF$. So
there is an $\Rr$-algebra isomorphism
\[
(\DF^\star,\circ) \simeq (\DF,\cdot)
\]
\end{thm}
\begin{proof}
By Corollary~\ref{thm:endpres} the ring $(\hh_{B^2}(G),\circ)\simeq (\hh_{B}(X),\circ)$ is isomorphic to the subalgebra of $\End_\Ss(\hh_B(X))$ generated by the image of the forgetful map $\hh_G(X)\to\hh_B(X)$ and push-pull operators $q_i^*q_{i*}$. Since the map $B\to B/T$ is an affine fibration, the natural map $\hh_B(X)\to\hh_T(X)$ is an isomorphism. Hence we may identify $\Ss$ with $\hh_T(k)$ and $\End_\Ss(\hh_B(X))$ with $\End_\Ss(\hh_T(X))$. Observe that these identifications preserve push-pull operators. The inclusion of $T$-fixed point set $W\to X$ gives an embedding $\hh_T(X)\to\hh_T(W)=\Ss_W^\star\subseteq \Qq_W^\star$. By~\cite[Corollary 8.7]{CZZ2} there is the following commutative diagram
\begin{equation}\label{eq:mdiag}
\xymatrix{
\hh_T(X)\ar[r]\ar[d]_{q_i^*q_{i*}} & \Ss_W^\star\ar@{^{(}->}[r] & \Qq_W^\star\ar[d]^{A_i}\\
\hh_T(X)\ar[r] & \Ss_W^\star\ar@{^{(}->}[r] & \Qq_W^\star
}
\end{equation}
where the Hecke operator $A_i$ is given by \[
A_i(f)(x)=f(x\cdot Y_i)\quad for\; x\in \Qq_W, f\in \Qq_W^\star.
\] 
Moreover, the forgetful map
\[\Ss\cong\hh_G(X)\to\hh_T(X)=\oplus_{w\in W}\Ss\]
is given by the formula $s\mapsto (w\cdot s)_{w\in W}$ for any $s\in\Ss$.
Then the multiplication in $\hh_T(X)=\Ss_W^*$ by the image of any element in $s\in\hh_G(X)$ induces a right multiplication by $s$ in $\Qq_W^*$.
Since $\Qq_W$ is a free $\Qq$-module of finite rank, the natural map $\imath\colon \Qq_W\to \End_\Qq(\Qq_W^\star)$ given by $\imath(x)(f)(y)=f(yx)$ is an inclusion. Note that every $A_i$ lies in the image of~$\imath$. Then by diagram \eqref{eq:mdiag} the image of $\hh_{B^2}(G)$ is isomorphic to a subalgebra of $\Qq_W$ generated by $\Ss$ and $Y_i$ which is $\DF$.
\end{proof}

\section{The rational algebra of push-pull operators} \label{sec:ratalg}

In the present section we introduce the rational algebra of push-pull operators $\DFR $ (Definition~\ref{def:ratalg}) and show that
it can be identified with the subring of rational endomorphisms of $G/B$ (Theorem~\ref{cor:surmain}).

\medskip

The $B^2$-equivariant isomorphism $E\times_k G\to E\times_k E$, $(e,g)\mapsto(e,eg)$ induces an isomorphism
$E\times^B G/B\to E/B\times_k E/B$. For all $w\in W$ fix a reduced decomposition $I_w=(i_1,\ldots,i_l)$ and the corresponding Bott-Samelson resolution
$X_{I_w}\to G/B$ of the Schubert cell. This map is $B$-equivariant, so it descends to a map
$Y_{I_w}=E\times^BX_{I_w}\to E\times^B G/B$.

\begin{lem}\label{lem:prebott}
The classes $[Y_{I_w}]$ form a basis of $\hh(E/B\times_k E/B)$ over $\hh(E/B)$, where the module structure is given by the pullback of the projection $pr_1^*\colon\hh(E/B)\to\hh(E/B\times_k E/B)$.
\end{lem}

\begin{proof}
Since $B$ is special, $G$-torsor $E$ splits over the function field of $E/B$. 
Then by \cite[Lemma~3.3]{PSZ} projection $pr_1\colon E/B\times_k E/B\to E/B$ is a cellular fibration in the sense of \cite[Definition 3.1]{PSZ} so that $(E/B)^2$ is a cellular space over $E/B$.
Let $\xi$ be the generic point of $E/B$.
The pullback of an open embedding \[j^*\colon\hh(E/B\times_k E/B)\to\hh(\xi\times_k E/B)\simeq\hh(G/B)\] is surjective and any preimage of $\Rr$-basis of $\hh(G/B)$ gives a basis of $\hh(E/B\times_k E/B)$. Thus it is sufficient to check that $j^*$ sends $[Y_{I_w}]$ to a basis of $\hh(\xi\times_k E/B)$. Let $p\colon E\to E/B$ be the projection. Note that \[E\times^BX_{I_w}\times_{(E/B\times_k E/B)}\xi\times_k E/B=p^{-1}(\xi)\times^BX_{I_w}=\xi\times X_{I_w},\] since $p^{-1}(\xi)\to\xi$ is a trivial $B$-torsor. Thus $j^*([Y_{I_w}])=[\xi\times X_{i_w}]$ forms a basis of $\hh(\xi\times E/B)=\hh(\xi\times G/B)$ over $\hh(\xi)=\Rr$.
\end{proof}
Consider a $B$-equivariant map \[
f\colon E\times^B G\to B\backslash G,\quad (e,g)B\mapsto Bg.
\] Let $X'_{I_w}=(P_{i_1}\times\ldots\times P_{i_l})/B^l$ where $B^l$-action on $P_{i_1}\times\ldots\times P_{i_l}$ is given by $(p_1,\ldots, p_l)\cdot(b_1,\ldots, b_l)=(b_1^{-1}p_1b_2,\ldots,b_l^{-1}p_l)$. Then $X'_{I_w}$ gives the Bott-Samelson class for $B\backslash G$.
\begin{lem}\label{lem:bott}
The composition $\hh_B(B\backslash G)\stackrel{f^*}\to\hh_B(E\times^B G)\simeq\hh(E/B\times_k E/B)$ maps $[X'_{I_w}]_B$ to $[Y_{I_w}]$.
\end{lem}
\begin{proof}
Consider the map $P_{i_1}\times^BP_{i_2}\times^B\ldots\times^B P_{i_l}\to G$ given by $(p_1,\ldots, p_l)\to p_1\ldots p_l$. It is $B$-equivariant with respect to the left multiplication, so it descends to a map $M_{I_w}=E\times^BP_{i_1}\times^BP_{i_2}\times^B\ldots\times^B P_{i_l}\to E\times^B G$. By construction we have an isomorphism 
\[M_{I_w}\simeq Y_{I_w}\times_{E\times^B(G/B)}(E\times^B G).\] Then $[M_{I_w}]_B$ is mapped to $[Y_{I_w}]$ via the isomorphism $\hh(E\times^B G/B)\to\hh_B(E\times^B G)$. Thus it is sufficient to check that $f^*[X'_{I_w}]_B=[M_{I_w}]_B$, which follows from the fact that 
\[M_{I_w}=E\times^B(P_{i_1}\times\ldots\times P_{i_l}/B^{l-1})\simeq (E\times^B G)\times_{B\backslash G}X'_{I_w}.\qedhere\]
\end{proof}

\begin{lem}(cf. \cite[Corollary~3.4]{PSZ})\label{lem:surjc} The composition
\[
(p_1^*,\gamma_1^*)\colon  \hh_B(E) \otimes_{\Ss} \hh_{B^2}(G) \longrightarrow  \hh_{B^2}(E^2)\simeq \hh((E/B)^2) 
\] of the diagram~\eqref{diag:basech} (for $H=B$) is an isomorphism. 
\end{lem}

\begin{proof} 
Consider the basis of $\hh_{B^2}(G)$ over $\Ss$ given by the classes of Bott-Samelson resolutions $\zeta_{I_w}$.
Then by Lemma~\ref{lem:bott} $\gamma_1^*(\zeta_{I_w})$ forms a basis of $\hh_{B^2}(E^2)$ over $\hh_B(E)$ induced by the respective cellular filtration.
\end{proof}

Consider the restriction map $\hh(E/B)\to \hh(E_{\kb}/B)=\hh(X_{\kb})$ on cohomology induced by the scalar extension $\kb/k$ (here $\kb$ is a splitting field of $E$).
Let $\hhr(X)$ denote its image. 

\begin{cor}\label{thm:rational}
The image of the ring homomorphism
\[
res_{\kb/k}\colon (\hh(E/B\times_k E/B),\circ) \longrightarrow (\hh(X_{\kb}\times_{\kb} X_{\kb}),\circ).
\]
is the subalgebra generated by the multiplication by the elements of $\hhr(X)$ and the push-pull operators 
$q_i^*q_{i*}\colon\hh(X)\to\hh(G/P_i)\to\hh(X)$ for all simple roots $\alpha_i$. 
\end{cor}

\begin{proof}
Follows by \eqref{diag:basech}, Lemma~\ref{lem:surjc} and Corollary~\ref{thm:endpres}.
\end{proof}

There is a natural action of $W$ on $\hhr(X)$ that comes from the $W$-action on $E/T$. So we can endow $\hhr(X)\otimes_\Ss \Qq_W$
with a structure of an $\Rr$-algebra. 

\begin{dfn}\label{def:ratalg}
Let $\DFR $ denote its subalgebra $\hhr(X)\otimes_\Ss\DF$. We call it the rational algebra
of push-pull operators. By $\DFR^{(m)}$ we denote its degree $m$ homogeneous component assuming that all $Y_i$'s have degree $(-1)$  and elements of $\hhr^\bullet(X)$ (and of $\Ss=\hh^\bullet_B(k)$) have degree '$\bullet$'. 
\end{dfn}

Observe that if $E$ is split, then $\DFR=\hh(X)\otimes_\Ss \DF$ does not coincide with $\DF$.

Set $N=\dim X$.

\begin{thm}\label{cor:surmain} Consider the restriction 
\[
res_{\kb/k}\colon \End_{\chm(k)}([E/B]) \longrightarrow \End_{\chm(\kb)}([X_{\kb}])
\]
on endomorphism rings of the respective motives (i.e., preserving the grading of $\hh(X)$).
Its image can be identified with $\DFR^{(0)}$ via the injective forgetful map
\[
\phi\colon \big((\hhr(X)\otimes_\Ss \hh_G(X_{\kb}^2))^{(N)},\circ\big) \longrightarrow \big(\hh^{N}(X_{\kb}^2),\circ\big).
\]
\end{thm}

\begin{proof}
By~\eqref{diag:comconst} both $\hh_G(X^2)$ and $\hh(X^2)$ are free modules over $\hh_G(X)$ and $\hh(X)$ with basis given by the classes $[\Ot_{I_w}]_G$ and $[\Ot_{I_w}]$ respectively. The map $\phi$ sends $[\Ot_{I_w}]_G\mapsto [\Ot_{I_w}]$ and 
leaves the coefficients invariant.
The result follows by Corollary~\ref{thm:rational}, Corollary~\ref{thm:endpres} and Theorem~\ref{thm:Demazure}.
\end{proof}

We say that a (co-)homology theory $\hh$ satisfies the Dimension Axiom if
\begin{itemize}
\item[(Dim)] For any smooth variety $Y$ over $k$ we have $\hh^{n}(Y)=0$ for all $n>\dim Y$.
\end{itemize}

\begin{ex}
Any theory $\hh$ over a field $k$ of characteristic $0$ obtained by specialization of coefficients of the Lazard ring (e.g. Chow groups, connective $K$-theory, algebraic cobordism $\Omega$) satisfies (Dim). 
The graded $K$-theory $K_0(-)[\beta,\beta^{-1}]$ of \cite[Example~1.1.5]{LM} does not satisfy (Dim).
\end{ex}

\medskip

Observe that the image of the characteristic map $c\colon \Ss\to \hh(X)$ is contained
in $\hhr(X)$ (see \cite[Thm.~4.5]{GZ}). Consider both the induced map $\mathfrak{c}\colon \DF \to \DFR$ and the restriction map $res_{\kb/k}\colon (\hh_{B}(E)\otimes_\Ss \DF) \to \DFR$. We will use the following substitute of the Rost nilpotence theorem.

\begin{lem}\label{lem:kernil}
Assume that the theory $\hh$ satisfies (Dim), then the kernels of $\mathfrak{c}$ and $res_{\kb/k}$ are complete.
In other words,
there is a commutative diagram of maps of convolution rings
\[
\xymatrix{
\DF \ar[r]^-{\gamma_1^*} \ar[rd]_{\mathfrak{c}}& \hh((E/B)^2)\ar@{>>}[d]^{res_{\kb/k}} \\
 & \DFR}
\]
with complete kernels (cf.\cite[Ch.2, Lemma 2.2]{Wei}).
\end{lem}

\begin{proof} If $\mathfrak{c}(x)=0$ (resp. $res_{\kb/k}(x)=0$) then $x=\sum_w a_w Y_{I_w}$ with $a_w\in \Ss$, $c(a_w)=0$ (resp. with $a_w\in\hh(E/B)$, $res_{\kb/k}(a_w)=0$).
Since $c$ (resp. $res_{\kb/k}$) commutes with push-pull operators on $\Ss$ (resp. on $\hh(E/B)$), the product of $n$ such elements
$x_1\circ x_2\circ\ldots\circ x_n$ corresponds to 
\[
(\sum_w a_w^1 Y_{I_w})\ldots (\sum_w a_w^n Y_{I_w})=\sum_w a_{w,n} Y_{I_w}, \quad a_{w,n} \in (\ker c)^n,\;\text{(resp. }(\ker res_{\kb/k})^n)
\]

Since $\ker c$ is contained in the augmentation ideal, and $\Ss$ is complete, then the series $\sum_n a_{w,n}$ converges in $\Ss$ (resp. $\ker res_{\kb/k}$ is contained in $\hh^{>0}(E/B)$, then $a_{w,n}=0$ for $n>\dim E/B$),  thus $\ker \mathfrak{c}$ is complete and $\ker res_{\kb/k}$ is nilpotent, hence complete.
\end{proof}

\begin{lem}\label{for:dim lift} 
If $E$ is a generic $G$-torsor in the sense of \cite[\S3, p.108]{GMS}, then the map $\gamma_1^*$ and, hence, $\mathfrak{c}$, of the lemma~\ref{lem:kernil} is surjective.
\end{lem}

\begin{proof}
Observe that if $E$ is generic, then it admits a generically open $G$-equivariant embedding
into $\mathbb{A}^N_k$. So 
the projection $E\times_k G^i\to G^i$ in the definition of $\gamma_i$ factors through $\mathbb{A}^N_k\times_k G^i$.
By (Loc) and (HI) the induced pullback $\gamma_i^*$ is surjective. 
\end{proof}

\section{Applications  and examples}

Our main application is the following

\begin{thm}\label{mainthm} Let $G$ be a split semisimple linear algebraic group over a field $k_0$ and let $E$ be a $G$-torsor over a field extension $k$ of $k_0$. Let $\hh$ be an oriented cohomology theory over $k$ that satisfies both (CD) and (Dim) axioms.
Let $\langle [E/B] \rangle_\hh$ (resp. $\langle [E/B]_* \rangle_\hh$) denote the pseudo-abelian subcategory generated by the (resp. non-graded) $\hh$-motive of $E/B$. Then
\begin{itemize}
\item[(i)] There are isomorphisms between the Grothendieck groups
\[
K_0(\langle [E/B] \rangle_\hh)\simeq K_0\big(\DFR^{(0)}\big)\quad\text{and}\quad K_0(\langle [E/B]_* \rangle_\hh)\simeq K_0\big(\DFR\big).
\]

\item[(ii)] There is a 1-1 correspondence between direct sum decompositions of the $\hh$-motive $[E/B]$ and direct sum decompositions
of the $\DFR^{(0)}$-module $\DFR^{(0)}$. Moreover, any two direct summands in the motivic decomposition of $E/B$ that are Tate twists of each other 
correspond to isomorphic $\DFR$-modules.

\item[(iii)] If $E$ is generic, then the algebra $\DFR$ in (i) and (ii) can be replaced by the algebra $\DF$.
\end{itemize}
\end{thm}

\begin{proof}Consider a graded endomorphism ring of the $\hh$-motive of $E/B$\[
C_F=(End_{\chm(k)}^\bullet ([E/B]),\circ).\] 
Our main result (Theorem~\ref{cor:surmain}) together with Lemma~\ref{lem:kernil} says that the restriction map gives a surjective ring homomorphism  with complete kernel
\[
res_{\kb/k}\colon C_F \to \DFR.
\]
The part (i) then follows from \cite[Ch.~2, Lemma~2.2]{Wei}.

The part (ii) follows from \cite[\S2 and Prop.~2.6]{PSZ} applied to the map $res_{\kb/k}$ 
(an isomorphism between Tate twists in the non-graded category of motives corresponds to an isomorphism between idempotents of \cite[\S2.1]{PSZ}).

The part (iii) follows from Lemma~\ref{lem:kernil} applied to the map $\mathfrak{c}$.
\end{proof}

\begin{ex}
If $G$ is special, i.e. $G$ is simple simply-connected of type $A$ or $C$, then any $G$-torsor $E$ splits, and the respective non-graded subcategory $\langle [E/B]_*\rangle_\hh$ is generated by the motive of a point. So by (ii) and (iii) of the theorem all indecomposable direct summands of $\DFR$ (resp. of $\DF$) are isomorphic to the 
$\DFR$-(resp. $\DF$-)module $\Ss$ and 
\[
K_0(\DFR)\simeq K_0(\DF)\simeq \Z.
\]
\end{ex}

\begin{lem}
If the coefficient ring $\Rr$ is Artinian, then both $C_F$ and $\DFR$ and, hence, the categories $\langle [E/B]_*\rangle_\hh$ and $Proj\; \DFR$ satisfy the Krull-Schmidt property (uniqueness of a direct sum decomposition).
\end{lem}

Observe that in general the ring $\DFR$ is not Krull-Schmidt (and not semi-simple).

\begin{proof}
If $\Rr$ is Artinian, then both $\DFR$ and $C_F$ are Artinian (as $\DFR$ is finite dimensional over $\Rr$). So they are both Noetherian which implies that the respective tautological modules $\DFR$ and $C_F$ have finite length and, hence, the Krull-Schmidt property holds for both  $\DFR$ and $C_F$. 
\end{proof}

As a direct application of the main result of \cite{PSZ} one obtains the following characterization of modular representations of the (affine) nil-Hecke algebra 
($F$ is an additive formal group law and $\hh(-)=CH(-;\F_p)$).

\begin{cor}\label{cor:nilHeckeex} Let $G$ be a split semisimple linear algebraic group over a field $k_0$.
Consider the affine nil-Hecke algebra
$\mathbf{H}_{nil,p}$ for $G$ with coefficients in $\Rr=\F_p$, $p$ is a prime.
Then
\[
K_0\big(\mathbf{H}_{nil,p}\big) \simeq K_0(\langle \Ro_* \rangle),
\]
where $\Ro_*$ is the generalized (non-graded) Rost-Voevodsky motive corresponding to a generic $G$-torsor $E$ and the prime $p$.

In particular, all indecomposable graded submodules of $\mathbf{H}_{nil,p}$ 
are isomorphic to a graded indecomposable submodule $P_*$ corresponding to $\Ro_*$. Moreover, they are free $\Ss$-modules of rank $r$ that is equal to the $p$-part of the product
of $p$-exceptional degrees of $G$ and we have
\[
\mathbf{H}_{nil,p} \simeq Mat_{|W|/r} (\End(P_*)).
\]
\end{cor}

\begin{proof}
The $\Ss$-rank coincides with the number of Tate motives in the decomposition of $\Ro$ over a splitting field of $E$, that is $g_p(1)=\prod_{i=1}^r \tfrac{1-t^{d_ip^{k_i}}}{1-t^{d_i}}|_{t=1}$ (in the notation of \cite{PSZ}) which is equal to the $p$-part $p^{\sum_{i=1}^r k_i}$ of $p$-exceptional degrees of \cite[p.73]{Ka}.
The last statement follows from the fact that the ring of graded endomorphisms of $\Ro_*$ coincides with the ring of endomorphisms of $\Ro$.
\end{proof}

We now switch to integer coefficients. Recall that in this case the Krull-Schmidt property usually fails.  

\begin{ex}
Consider the root system of type $A_1$. In this case $T^*=\Z \omega$  or $T^*=\Z\alpha$, $\alpha=2\omega$ is the simple root and $\omega$ is the fundamental weight. The Weyl group $W=\{1,s\}$ acts by $s\colon \omega\mapsto -\omega$, where $s$ is the simple reflection. By definition, $\Ss=\Rr[[x]]_F$ (where $x=x_\omega$ or $x=x_\alpha$), $\Qq=\Ss[\tfrac{1}{x}]$, $\Qq_W=\{q(x)\delta_w\mid q(x)\in \Qq,\; w\in W\}$ with \[
q(-_Fx)\delta_s=s(q(x))\delta_s=\delta_s q(x),
\] where $-_F x$ is the formal inverse of $x$. Observe that $x_\alpha=x_{\omega+\omega}=F(x_\omega,x_\omega)$ in $\Ss$. 

The $\Rr$-algebra $\DF$ 
is a free left $\Ss$-submodule of rank 2 in $\Qq_W$ with basis \[\{1,Y=\tfrac{1}{-_Fx_\alpha}+\tfrac{1}{x_\alpha}\delta_s\}.\]
It satisfies the relations
\[
Y^2=\kappa Y\text{ and }Yq(x)=q(-_F x)Y+\Delta(q(x)),
\]
where $\kappa=\tfrac{1}{-_F x_\alpha}+\tfrac{1}{x_\alpha}$ and  $\Delta(q(x))=\tfrac{q(x)-q(-_F x)}{-_F x_{\alpha}}$.

Let $p=a+bY$, where $a,b\in \Ss$, be an idempotent in $\DF$, i.e., $p^2=p$. 
Then we obtain in $\Ss$
\begin{equation}\label{idem}
a^2+b\Delta(a)=a\;\text{ and }\;(a+s(a)+s(b)\kappa+\Delta(b))b=b.
\end{equation}

In the case $\hh(-)=CH(-;\Z)$  ($F$ is additive) we have $\Rr=\Z$, $\Ss=\Z[x]$, $\kappa=0$, $-_F x=-x$, $x_\alpha=2x_\omega$ and 
the second equation turns into
\[
(a+s(a)+\Delta(b))b=b.
\]

If $x=x_\alpha$ ($G=PGL_2$), the polynomials $a+s(a)$ and $\Delta(b)$ are divisible by $2$, so $b=0$. Hence, $a=0$ or $1$ from the first equation.
So $\DF$ is an indecomposable module over itself.
By Theorem~\ref{mainthm} this implies that both graded and non-graded
motives $[E/B]$ and $[E/B]_*$ of a generic conic are indecomposable.

If $x=x_\omega$ ($G=SL_2$) and $p$ is homogeneous, we get
$a\in \Z$, $b=cx$, $c\in \Z$ and the system \eqref{idem} has solutions only for $c=\pm 1$.
Therefore, the algebra $\DF$ has only two indecomposable graded submodules which correspond to the idempotents $1-x_\omega Y$ and $x_\omega Y$.
In other words, the motives $[E/B]$ and $[E/B]_*$ split into a direct sum of two indecomposable motives.
\end{ex}

\bibliographystyle{plain}

\end{document}